\documentclass{article}
\usepackage[utf8]{inputenc}
\usepackage{amsmath}
\usepackage{amssymb}
\usepackage{float}

\usepackage{graphicx}

\usepackage{
	amsmath,			% Math Environments
  	amssymb,			% Math Symbols
  	amsthm,			    % Theorem Environments 
  	bookmark,        	% Replaces (and loads) hyperref
	chngcntr,			% Use counterwithin
	enumerate,		    % Enumerate Environments
	fancyhdr,			% Better Headers
	float,				% Force Image Placement
  	mathrsfs,           % More Math Scripts
  	mathtools,       	% Allows Prescript
  	microtype,       	% Improved Appearance
	multicol,			% Use Multi-columns
	multirow,			% Use Multi-rows
	nicefrac,			% Alternative Fractions
	relsize,			% Bigger Math Symbols
	stmaryrd,			% Power Series Brackets
  	thmtools,           % Custom Theorem Environments
	thm-restate,		% For Re-stateable Theorems
	totcount,			% Define counters
	braket,             % For \set command
	commath,             % For \abs command
}

\usepackage[dvipsnames]{xcolor}
\usepackage{hyperref}
\hypersetup{
    colorlinks=true,
    linkcolor=blue,
    citecolor=blue
}

\usepackage{tikz}
\usetikzlibrary{decorations.pathreplacing, arrows,positioning}

\usepackage[T1]{fontenc}
\usepackage{footnote,pifont,dsfont}
\usepackage[mathscr]{eucal}
\usepackage{fullpage}

\newcommand{\mA}{\ensuremath{\mathcal{A}}}

\newcommand{\mD}{\ensuremath{\mathcal{D}}}
\newcommand{\mB}{\ensuremath{\mathcal{B}}}

\newcommand{\mR}{\ensuremath{\mathcal{R}}}
\newcommand{\mS}{\ensuremath{\mathcal{S}}}

\newcommand{\mN}{\ensuremath{\mathcal{N}}}
\newcommand{\mM}{\ensuremath{\mathcal{M}}}

\newcommand{\mV}{\ensuremath{\mathcal{V}}}
\newcommand{\sing}{\ensuremath{\mathcal{V}}}

\newcommand{\C}{\mathbb{C}}
\newcommand{\Z}{\mathbb{Z}}
\newcommand{\R}{\mathbb{R}}
\newcommand{\N}{\mathbb{N}}

\newcommand{\bi}{\ensuremath{\mathbf{i}}}

\newcommand{\ww}{\ensuremath{\mathbf{w}}}

\newcommand{\bw}{\ensuremath{\mathbf{w}}}

\newcommand{\xx}{\ensuremath{\mathbf{x}}}

\newcommand{\bz}{\ensuremath{\mathbf{z}}}
\newcommand{\zz}{\ensuremath{\mathbf{z}}}
\newcommand{\one}{\ensuremath{\mathbf{1}}}

\newcommand{\oS}{\ensuremath{\overline{S}}}

\newcommand{\ow}{\ensuremath{\overline{w}}}

\newcommand{\oz}{\ensuremath{\overline{z}}}

\newcommand{\grad}{\ensuremath{\nabla}}

\newcommand{\bzht}[1]{\bz_{\hat{#1}}}

\newcommand{\htbw}{\widehat{\bw}}
\newcommand{\htbz}{\widehat{\bz}}

\newcommand{\bsigma}{\ensuremath{\boldsymbol \sigma}}

\newcommand{\brho}{\ensuremath{\boldsymbol \rho}}

\newcommand{\btau}{\ensuremath{\boldsymbol \tau}}
\newcommand{\bzer}{\ensuremath{\mathbf{0}}}
\newcommand{\btheta}{\ensuremath{\boldsymbol{\theta}}}
\newcommand{\bone}{\ensuremath{\mathbf{1}}}

\makeatletter
\def\testb#1{\testb@i#1,,\@nil}%
\def\testb@i#1,#2,#3\@nil{%
  \draw[->, thick] (O) --++(#1);
  \ifx\relax#2\relax\else\testb@i#2,#3\@nil\fi}
\makeatother   
\newcommand{\makediag}[1]{
    \coordinate (O) at (0,0); \coordinate (N) at (0,0.8); \coordinate (N2) at (0,0.7);
    \coordinate (NE) at (0.8,0.8); \coordinate (E) at (0.8,0);
    \coordinate (SE) at (0.8,-0.8); \coordinate (S) at (0,-0.8);
    \coordinate (SW) at (-0.8,-0.8);\coordinate (W) at (-0.8,0);
    \coordinate (NW) at (-0.8,0.8); \coordinate (B1) at (1.2,1.2);
    \coordinate (B2) at (-1.2,-1.2);
    \testb{#1}
} 
\newcommand{\diagr}[1]{
  \begin{tikzpicture}[scale=0.8]\makediag{#1}\end{tikzpicture}
}

\renewcommand{\exp}[1]{\operatorname{exp} \left[ #1 \right]}

\newtheorem{theorem}{Theorem}

\newtheorem{proposition}[theorem]{Proposition}

\newtheorem{lemma}[theorem]{Lemma}

\theoremstyle{definition}

\theoremstyle{remark}
\newtheorem{rem}[theorem]{Remark}

\newtheorem{examplex}[theorem]{Example}
\newenvironment{example}
  {\pushQED{\qed}\examplex}
  {\popQED\endexamplex}

\title{Completing the Asymptotic Classification of Mostly Symmetric Short Step Walks in an Orthant}
\author{Alexander Kroitor and Stephen Melczer}
\date{}

\begin{document}

\maketitle

\begin{abstract}
In recent years, the techniques of analytic combinatorics in several variables (ACSV) have been applied to determine asymptotics for several families of lattice path models restricted to the orthant $\N^d$ and defined by step sets $\mS\subset\{-1,0,1\}^d\setminus\{\mathbf{0}\}$. Using the theory of ACSV for smooth singular sets, Melczer and Mishna determined asymptotics for the number of walks in any model whose set of steps $\mS$ is `highly symmetric' (symmetric over every axis). Building on this work, Melczer and Wilson determined asymptotics for all models where $\mS$ is `mostly symmetric' (symmetric over all but one axis) \emph{except} for models whose set of steps have a vector sum of zero but are not highly symmetric. In this paper we complete the asymptotic classification of the mostly symmetric case by analyzing a family of saddle-point-like integrals whose amplitudes are singular near their saddle points. 
\end{abstract}

Lattice path enumeration is a classical problem in enumerative combinatorics. Given a (typically finite) set of steps $\mS \subset \Z^d$, a \emph{restricting region} $\mR \subset \Z^d$, a set of \emph{starting positions} $\mA \subset \mR$, and a set of \emph{ending positions} $\mB \subset \mR$, the goal is to enumerate the number of walks taking $n$ steps in $\mS$ that start in $\mA$, end at $\mB$, and always stay in $\mR$. Lattice path enumeration has a long history in combinatorics and probability theory (see, for instance, the textbooks~\cite{Mohanty1979,Narayana1979,KrattenthalerMohanty2003} and the survey~\cite{Humphreys2010}), and finds application to a wide variety of problems in combinatorics and broader fields. There are a vast number of approaches to the enumeration of walks in convex cones: a full accounting would take up its own survey paper, but to illustrate the breadth of work on the subject we note (in addition to the analytic viewpoint of this paper) techniques involving computer algebra~\cite{BostanKauers2009,KauersKoutschanZeilberger2009,BostanRaschelSalvy2014,BostanChyzakHoeijKauersPech2017}, differential Galois theory~\cite{DreyfusHardouin2021,DreyfusHardouinRoquesSinger2021}, potential theory~\cite{DArcoLacivitaMustapha2016}, boundary value problems on Riemann surfaces~\cite{Raschel2012}, probabilistic methods~\cite{FayolleIasnogorodskiMalyshev2017,DenisovWachtel2015}, and elegant power series manipulations~\cite{Bousquet-Melou2016,BernardiBousquet-MelouRaschel2021}. More recent work has also studied walks in non-convex cones~\cite{RaschelTrotignon2019, DreyfusTrotignon2021, Price2022, Trotignon2022, Bousquet-Melou2023}.

Much attention in the enumeration literature has focused on walks restricted to an orthant $\mR=\N^d$. Using the \emph{kernel method}, an enumerative technique used to manipulate functional equations satisfied by lattice path generating functions, it is often possible to represent generating functions enumerating walk models as \emph{diagonal extractions} of explicit multivariate series (see~\cite[Section 4.1]{Melczer2021} for an overview of the kernel method, and~\cite[Section 4.2]{Melczer2021} for a historical perspective). Diagonal extractions can be combined with the theory of \emph{Analytic Combinatorics in Several Variables} (ACSV)~\cite{Melczer2021,PemantleWilsonMelczer2024} to determine asymptotics for lattice path models. Given a $(d+1)$-variate power series\footnote{Throughout this paper we use the multi-index notation $\bz^{\bi} = z_1^{i_1}\cdots z_d^{i_d}$.}
\[ F(\zz,t) = \sum_{(\bi,t) \in \N^{d+1}} f_\bi\bz^\bi t^n, \]
the \emph{(main) diagonal} of $F$ is the univariate series
\[ (\Delta F)(t) = \sum_{n \geq 0}f_{n,\dots,n}t^n \]
defined by the coefficients in $F(\zz,t)$ where all exponents are equal. The techniques of ACSV show how the analytic behaviour of $F$ near its singular set describes the asymptotic behaviour of its diagonal sequence. In this paper we represent the generating functions for a family of lattice path models as diagonal extractions of multivariate rational functions, then use ACSV to determine asymptotics.

\section{Models with Highly and Mostly Symmetric Step Sets}

Fix a dimension $d\in\mathbb{N}$ and a step set $\mS \subset \{\pm1,0\}^d \setminus \{\bzer\}$. Walk models whose steps have coordinates equal to $0$ or $\pm1$ are known as \emph{short step models}. To rule out redundant cases, we assume that for each $1 \leq j \leq d$ there is a step in $\mS$ moving forward in the $j$th coordinate, and a step in $\mS$ moving backwards in the $j$th coordinate. We also allow our steps to have positive weights, and define the \emph{weight} of a lattice walk to be the product of the weights of its steps. The \emph{(weighted) characteristic polynomial} of $\mS$ is the Laurent polynomial
\[ S(\bz) = \sum_{\bi \in \mS} w_{\bi} \bz^\bi \]
whose exponents encode the entries in $\mS$, where each $w_{\bi}$ is a positive real weight. Define the notation
\[ \bz_{\hat{k}} = (z_1,\dots,z_{k-1},z_{k+1},\dots,z_d) \]
for any $1 \leq k \leq d$ and $\htbz = \bz_{\hat{d}} = (z_1,\dots,z_{d-1})$. Furthermore, following convention in the lattice path literature, let $\overline{v}=1/v$ for any variable $v$ (we do not refer to complex conjugation in this paper).

\begin{rem} 
In the \emph{unweighted} case when each weight has the value 1, the value $S(\bone) = |\mS|$ is the cardinality of the step set and the total weight of walks of length $n$ equals the number of paths of length $n$. When each weight is a positive integer then we can imagine counting paths where there are (potentially) multiple copies of each step. 
\end{rem}

We say $\mS$ is \emph{symmetric} over a coordinate axis if $\mS$ is unchanged by reflection over the axis \emph{and} the weight of any step equals the weight of the step obtained by reflecting over the axis. In this paper we restrict to the cases where $\mS$ is either symmetric over every coordinate axis or all but one axis. We may assume without loss of generality that the axis of non-symmetry (if it exists) is $z_d$, so that
\[ S(\bz) = \oz_d A\left(\htbz\right) + Q\left(\htbz\right) + z_d B\left(\htbz\right)   \]
for Laurent polynomials $A,B,$ and $Q$ that are symmetric in their variables. For all $1\leq k \leq d$ let $b_k$ be the total weight 
\[ b_k = \sum_{\substack{\bi \in \mS \\[+0.5mm] i_k=1}} w_{\bi}  \]
of the steps moving forward in the $k$th coordinate.

\begin{theorem}[{Highly Symmetric Asymptotics~\cite[Theorem 3.4]{MelczerMishna2016}}] \label{thm:MeMiAsm}
Let $\mS \subset \{-1,0,1\}^d \setminus \{\mathbf{0}\}$ be a set of steps that is symmetric over every axis and moves forwards and backwards in each coordinate. Then the total weight~$s_n$ of walks of length~$n$ taking steps in~$\mS$, beginning at the origin, and never leaving $\N^d$ satisfies
\[
s_n  = S(\one)^n \cdot n^{-d/2} \cdot  \left[\, \left(\frac{S(\one)}{\pi}\right)^{d/2} \frac{1}{\sqrt{b_1 \cdots b_d}} + O\left( \frac{1}{n} \right)\, \right].
\]
\end{theorem}

Theorem~\ref{thm:MeMiAsm} is obtained by applying the techniques of ACSV to the diagonal expression
\begin{equation} 
W(t) = \Delta \left(\frac{G(\bz,t)}{H(\bz,t)}\right) = \Delta\left(\frac{(1+z_1)\cdots(1+z_d)}{1 - t(z_1\cdots z_d)S(\bz)}\right),
\label{eq:highlyDiag}
\end{equation}
where $W(t)$ is the generating function enumerating walks in the model defined by $\mS$. We note that the singular set $\mV$ of the rational function $G/H$ are the zeroes of the polynomial $H(\zz,t) = 1 - t(z_1\cdots z_d)S(\bz)$. Because this polynomial and its partial derivative with respect to $t$ never simultaneously vanish, $\mV$ forms a manifold and only the results of \emph{smooth} ACSV (the simplest case) are needed.

Asymptotics in the mostly symmetric case depends on the \emph{drift} $B(\bone) - A(\bone)$ of a walk with respect to the $z_d$-axis, which is the weight of the steps in $\mS$ with positive $z_d$ coordinate minus the weight of the steps in $\mS$ with negative $z_d$ coordinate. In the positive drift case, the number of walks of length $n$ satisfies a formula similar to the highly symmetric case.

\begin{theorem}[{Positive Drift Asymptotics~\cite[Theorem 2.1]{MelczerWilson2019}}]
\label{thm:PosAsm} 
Let $\mS$ be a step set that is symmetric over all but the $d$th axis and takes a step forwards and backwards in each coordinate.  If $\mS$ has positive drift then the total weight $s_n$ of walks of length $n$ taking steps in $\mS$, starting at the origin, and never leaving $\N^d$ satisfies
\[
s_n = S(\bone)^n \cdot n^{-(d-1)/2} \cdot \left[\left(\frac{S(\bone)}{\pi}\right)^{\frac{d-1}{2}} \frac{B(\one) - A(\bone)}{B(\one)\sqrt{b_1 \cdots b_{d-1}}} + O\left(\frac{1}{n}\right)\right]. 
\]
\end{theorem}

The asymptotic behaviour of a model with a negative drift step set is slightly messier. Let $\rho = \sqrt{A(\bone)/B(\bone)},$ define
\[ B_k(\bzht{k}) = [z_k]S(\bz) = [z_k^{-1}]S(\bz) \]
for each $1 \leq k \leq d-1$, and let
\[ C_{\rho} = \frac{S(\bone,\rho) \, \rho}{2\, \pi^{d/2}\, A(\bone) (1-1/\rho)^2} \cdot \sqrt{\frac{S(\bone,\rho)^d}{\rho \, B_1(\bone,\rho) \cdots B_{d-1}(\bone,\rho) \cdot B(\bone)}}.\]
Furthermore, define the constant $C_{-\rho}$ by replacing $\rho$ with $-\rho$ in $C_{\rho}$ (the term under the square-root will always be real and positive when $C_{-\rho}$ is referenced). 

\begin{theorem}[{Negative Drift Asymptotics~\cite[Theorem 2.3]{MelczerWilson2019}}]
\label{thm:NegAsm}
Let $\mS$ be a negative drift step set that is symmetric over all but the $d$th axis and takes a step forwards and backwards in each coordinate. If $Q(\htbz) \neq 0$ (i.e., if there are steps in $\mS$ having $z_d$ coordinate $0$) then the total weight $s_n$ of walks of length $n$ taking steps in $\mS$, starting at the origin, and never leaving $\N^d$ satisfies
\[ s_n =  S(\bone,\rho)^n \cdot n^{-d/2-1} \cdot C_{\rho} \left(1 + O\left(\frac{1}{n}\right)\right) . \]
If $Q(\htbz) = 0$ then 
\[
s_n = n^{-d/2-1} \cdot {\Big[} S(\bone,\rho)^n \cdot C_{\rho} + S(\bone,-\rho)^n \cdot C_{-\rho} {\Big]}\left(1 + O\left(\frac{1}{n}\right)\right).
\]
\end{theorem}

Theorems~\ref{thm:PosAsm} and~\ref{thm:NegAsm} are derived by applying the techniques of ACSV to the diagonal expression
\begin{equation}
W(t) = \Delta\left( \frac{ G(\bz,t) }{H(\bz,t) } \right),
\label{eq:mostlyDiag}
\end{equation}
where
\begin{align*}
G(\bz,t) &= (1+z_1) \cdots (1+z_{d-1}) \left( 1- tz_1 \cdots z_d \left(Q(\htbz) + 2 z_d A(\htbz) \right) \right)\\
H(\bz,t) &= (1-z_d) \left( 1- tz_1 \cdots z_d \oS(\zz)\right) \left( 1- tz_1 \cdots z_d \left(Q(\htbz) +  z_d A(\htbz) \right) \right)
\end{align*}
for $\oS(\bz) = S(z_1,\dots,z_{d-1},\oz_d)$. In comparison to the diagonal expression~\eqref{eq:highlyDiag} for highly symmetric models, the denominator in~\eqref{eq:mostlyDiag} has multiple irreducible factors, and the singular set $\mV$ of $G/H$ is the union of three smooth surfaces. In the negative drift case the asymptotic behaviour of the diagonal sequence is still determined by the behaviour of $G/H$ near smooth points of $\mV$, however in the positive drift case the desired asymptotic behaviour is characterized by the behaviour of $G/H$ near non-smooth points of $\mV$. This geometric difference explains why positive and negative drift models have quantitatively different asymptotic behaviour (combinatorially, the difference is explained by the fact that positive drift models don't naturally wander into the restricting boundaries as frequently).

\begin{rem}
The techniques of ACSV allow the asymptotic expansions in Theorems~\ref{thm:MeMiAsm} to~\ref{thm:NegAsm} to be computed to any desired order. The expressions for higher-order constants are very unwieldy for general models, but can be computed automatically using a computer algebraic system for any explicit model.
\end{rem}

\subsection{Our Results}

Theorems~\ref{thm:MeMiAsm} to~\ref{thm:NegAsm} leave a gap, as they do not cover zero drift walks that are mostly symmetric but not highly symmetric. Our main result is to fill this gap.

\begin{theorem}
\label{thm:main}
Let $\mS$ be a step set that is symmetric over all but the $d$th axis and takes a step forwards and backwards in each coordinate.  If $\mS$ has zero drift then the total weight $s_n$ of walks of length $n$ taking steps in $\mS$, starting at the origin, and never leaving $\N^d$ satisfies
\[
s_n \sim S(\bone)^n \cdot n^{-d/2} \cdot \frac{ S(\one)^{d/2}}{\pi^{d/2} (b_1 \cdots b_d)^{1/2}}. 
\]
\end{theorem}

The dominant asymptotic behaviour of a zero drift mostly symmetric model matches the dominant term in the asymptotics for a highly symmetric model, however such mostly symmetric models only approximate highly symmetric models `up to the dominant term' and the subdominant terms do not match. In fact, mostly symmetric models have asymptotic expansions in powers of $n^{-1/2}$ compared to the highly symmetric models whose expansions are series in powers of $n^{-1}$ (see, for instance, Example~\ref{ex:zerodrift2} below). To the best of our knowledge, this is the first family of naturally occurring combinatorial classes analyzed using the methods of ACSV to not have a series in powers of~$n^{-1}$.

\begin{rem} 
In any dimension $d \geq 2$ there exist walk models with short step sets symmetric over all but two axes whose generating functions are non-D-finite, and therefore cannot be written as a rational diagonal (see~\cite[Theorem 3.8]{MelczerWilson2019}). Thus, mostly symmetric models are the largest class of short step models defined by axial symmetries that can be handled directly by applying the techniques of ACSV to rational diagonal expressions derived from the kernel method; see Table~\ref{tab:summary} for a (now complete) summary.
\end{rem}

\begin{rem}
Melczer and Wilson~\cite{MelczerWilson2019} incorrectly conjectured that the zero drift mostly symmetric models have asymptotic growth of the form $O(S(\one)^nn^{-(d+1)/2})$.
\end{rem}

\begin{table}
\centering
\begin{tabular}{|llllll|}
\hline
Drift & Exponential Rate & Order & Geometry & Series in & Covered By\\ \hline &&&&&\\[-3mm]
zero (highly sym.) &$S(\one)$ &$n^{-d/2}$ & smooth & $n^{-1}$ & Theorem~\ref{thm:MeMiAsm}\\
zero (mostly sym.) & $S(\one)$ & $n^{-d/2}$ & nonsmooth & $n^{-1/2}$ & Theorem~\ref{thm:main} \\
positive & $S(\one)$ & $n^{-(d-1)/2}$ & nonsmooth & $n^{-1}$ & Theorem~\ref{thm:PosAsm} \\
negative &$< S(\one)$  & $n^{-1-d/2}$ & smooth & $n^{-1}$ & Theorem~\ref{thm:NegAsm} \\
\hline
\end{tabular}
\caption{Summary of results for short step models with highly or mostly symmetric step sets.} 
\label{tab:summary}
\end{table}

Theorem~\ref{thm:main} is proven in Section~\ref{sec:mainProof}. First, in Section~\ref{sec:ACSV}, we give an overview of the methods of ACSV and describe why computing the behaviour of zero drift mostly symmetric models is harder than the other highly and mostly symmetric cases.

\section{ACSV and Lattice Path Enumeration}
\label{sec:ACSV}

The field of analytic combinatorics in several variables derives the asymptotic behaviour of the diagonal coefficients $(f_{n,\dots,n})$ of a power series expansion
\begin{equation} F(\zz,t) = \sum_{(\bi,n) \in \N^{d+1}} f_{\bi,n}\bz^\bi t^n \label{eq:ps} \end{equation}
by analyzing the singularities of the function $F$. Here we restrict ourselves to the case when $F(\zz,t)=G(\zz,t)/H(\zz,t)$ is a rational function defined by the coprime polynomials $G$ and $H$, so that the singular set $\sing$ of $F$ is defined by the vanishing of the denominator $H$.

Generalizing the more commonly known univariate setting, the starting point for ACSV is a Cauchy integral representation
\begin{equation} f_{n,\dots,n} = \frac{1}{(2\pi i)^d} \int_T F(\zz,t) \frac{d\zz}{(z_1\cdots z_d t)^{n+1}}, \label{eq:CIF} \end{equation}
where $T$ is a product of circles $|z_j|=\epsilon_j$ and $|t|=\epsilon$ lying in the interior of the domain of convergence $\mD$ of the power series~\eqref{eq:ps}. Many common univariate generating functions (for instance, all rational, algebraic, and D-finite functions) admit only a finite number of singularities, and classical analytic combinatorics shows how to manipulate Cauchy integral representations to determine the `contribution' of each singularity to asymptotic behaviour. In contrast, in the multivariate case the set $\sing$ is infinite and a more intricate analysis is required. Two particular types of singularities play a large role in the analysis. 

\begin{itemize}
	\item A \emph{minimal point} of $\sing$ is an element of $\sing$ on the boundary of $\mD$. Equivalently, $(\ww,t) \in \sing$ is a minimal point if and only if there does not exist $(\xx,s) \in \sing$ such that $|x_j|<|w_j|$ for all $1 \leq j \leq d$ and $|s|<|t|$. Minimal points are the singularities of $F$ that the domain of integration $T$ in~\eqref{eq:CIF} can be deformed arbitrarily close to by expanding each $|z_j|$ and $|t|$ independently. 
	\item A \emph{critical point} of $\sing$ is, informally, an element of $\sing$ around which $\sing$ has a saddle-point. The formal definition of a critical point depends on the geometry of $\sing$; when $H$ and all of its partial derivatives never simultaneously vanish (which happens generically) then the implicit function theorem implies that $\sing$ is a manifold and the critical points are defined by the polynomial system
	\[ H(\zz,t) = 0 \text{ and } z_1H_{z_1}(\zz,t) = \cdots = z_dH_{z_d}(\zz,t) = tH_{t}(\zz,t), \]
	where $H_v$ denotes the partial derivative of $H$ with respect to the variable $v$. In general, one partitions $\sing$ into a collection of smooth manifolds (a so-called \emph{Whitney stratification} of $\sing$) and computes the critical points on each manifold by solving a system of polynomial equalities and inequalities.
\end{itemize}

When there are minimal critical points, and other natural conditions hold, then one can deform the Cauchy domain of integration in~\eqref{eq:CIF} to lie close to these points (due to minimality), compute residues to reduce to lower-dimensional integrals lying `on' $\sing$, and then approximate these integrals using the saddle-point method (due to criticality). Generically there are a finite set of critical points, and the asymptotic contributions of each can be summed to give asymptotics of the diagonal sequence. For lattice path models, the analytic behaviour of $F$ near the points determining asymptotics is linked to symmetry properties of the corresponding step sets.

\subsection{Highly Symmetric Models}
The simplest case geometrically occurs for highly symmetric models. As mentioned above, there is a diagonal expression
\[ W(t) = \Delta \left(\frac{G(\bz,t)}{H(\bz,t)}\right) = \Delta\left(\frac{(1+z_1)\cdots(1+z_d)}{1 - t(z_1\cdots z_d)S(\bz)}\right),\]
for the generating function $W(t)$ enumerating walks in the model defined by $\mS$, where $S(\zz) = \sum_{\bi \in \mS}w_{\bi}\zz^{\bi}$. Because $H$ and its partial derivative with respect to $t$ never simultaneously vanish, the singular set $\sing$ is always a smooth manifold.

\begin{example}
\label{ex:smooth}
Consider the (unweighted) highly symmetric model with step set $\mS = \{(0,\pm1),(\pm1,0)$ defined by the cardinal directions
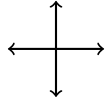
\[ \diagr{N,S,E,W} \]
whose generating function is the diagonal of 
\[ F(x,y,t) = \frac{G(x,y,t)}{H(x,y,t)} = \frac{(1+x)(1+y)}{1 - txyS(x,y)} \]
where $S(x,y)=x+1/x+y+1/y$. The system of polynomial equations
\[ H(x,y,t) = 0 \text{ and } xH_x(x,y,t) = yH_y(x,y,t) = tH_t(x,y,t) \]
defining the critical points has the two solutions $(1,1,1/4)$ and $(-1,-1,-1/4)$, both of which are minimal as if $|x|,|y|<1$ and $H(x,y,t)=0$ then $|t| = 1/|xyS(x,y)| > 1/4$.

The Cauchy integral formula implies that the number $s_n$ of walks on $\mS$ that start at the origin and stay in $\N^2$ satisfies
\[ s_n = \frac{1}{(2\pi i)^2} \int_{\substack{|x|=1 \\ |y|=1}} \left(\int_{|t| = 1/4-\epsilon} \frac{(1+x)(1+y)}{1 - txyS(x,y)} \frac{dt}{t^{n+1}} \right)  \, \frac{dxdy}{(xy)^{n+1}} \]
for any $\epsilon>0$ sufficiently small. Analytic arguments show that we can restrict the circles $|x|=|y|=1$ to sufficiently small neighbourhoods $\mN_{\pm}$ of $(1,1)$ and $(-1,-1)$, and expand the circle $|t|=1/4-\epsilon$ to a circle $|t|=1/4+\epsilon$, while introducing exponentially small errors. Because the numerator of $F$ vanishes at $(-1,-1)$ it turns out that the only point determining dominant asymptotic behaviour of $s_n$ is $(1,1)$, and we can write
\begin{align*} 
s_n &\sim \frac{-1}{(2\pi i)^2} \int_{\mN_+} \left(\int_{|t| = 1/4+\epsilon} \frac{(1+x)(1+y)}{1 - txyS(x,y)} \frac{dt}{t^{n+1}} - \int_{|t| = 1/4-\epsilon} \frac{(1+x)(1+y)}{1 - txyS(x,y)} \frac{dt}{t^{n+1}} \right)  \, \frac{dxdy}{(xy)^{n+1}} \\[+2mm]
&= \frac{1}{(2\pi i)^2} \int_{\mN_+} \frac{(1+x)(1+y)}{xy} S(x,y)^{n} \, dxdy,
\end{align*}
where the difference of integrals in the first line is computed by taking a residue at the pole $t=\frac{1}{xyS(x,y)}$. 

Parametrizing the neighbourhood $\mN_+$ of $(1,1)$ by $x=e^{i\theta_1}$ and $y = e^{i\theta_2}$ for $\theta_1,\theta_2$ in a neighbourhood $\mM$ of the origin, the argument concludes by applying the saddle-point method to approximate
\begin{align*}
\frac{1}{(2\pi i)^2} \int_{\mN_+} \frac{(1+x)(1+y)}{xy} S(x,y)^{n} \, dxdy
&= \frac{4^n}{(2\pi)^2} \int_{\mM} (1+e^{i\theta_1})(1+e^{i\theta_2}) e^{n \log \frac{S(e^{i\theta_1},e^{i\theta_2})}{S(1,1)}} d\theta_1d\theta_2 \\[+2mm]
&= \frac{4^n}{(2\pi)^2} \int_{\mM} (4 + O(\theta_1+\theta_2)) e^{-n(\theta_1^2/4 + \theta_2^2/4) + O(n(\theta_1 + \theta_2)^3)} d\theta_1d\theta_2 \\[+2mm]
&\sim \frac{4^n}{(2\pi)^2} \int_{\R^2} 4 e^{-n(\theta_1^2/4 + \theta_2^2/4)} d\theta_1d\theta_2 \\[+2mm]
&= \frac{4}{\pi} \cdot \frac{4^n}{n}.
\end{align*}
\end{example}

Melczer and Mishna~\cite{MelczerMishna2016} combined the techniques of smooth ACSV, which show how to make all steps in the above argument rigorous, with the uniform diagonal expression for highly symmetric models to prove Theorem~\ref{thm:MeMiAsm}.

\subsection{Mostly Symmetric Models}
The generating function $W(t)$ for walks in $\N^d$ on a set of mostly symmetric short steps $\mS$ has the diagonal representation
\[ W(t) = \Delta F(\zz,t) = \Delta\left(\frac{(1+z_1) \cdots (1+z_{d-1}) \left( 1- tz_1 \cdots z_d \left(Q(\htbz) + 2 z_d A(\htbz) \right)\right)}{(1-z_d) \left( 1- tz_1 \cdots z_d \oS(\zz)\right) \left( 1- tz_1 \cdots z_d \left(Q(\htbz) +  z_d A(\htbz) \right) \right)}\right), \]
and the three factors in the denominator give the singular set $\mV$ non-smooth points. The final denominator factor turns out to be irrelevant to the asymptotic analysis, so the points of interest lie on the zero sets $\sing_1$ and $\sing_2$ of $H_1(\zz,t) = 1-z_d$ and $H_2(\zz,t) = 1-tz_1\cdots z_d \oS(\zz)$. The sets $\sing_1$ and $\sing_2$ are each manifolds: $\sing_1$ does not contain any critical points, but $\sing_2$ contains the critical point 
\[ \brho = \left(\one,\sqrt{\frac{B(\one)}{A(\one)}},\frac{\sqrt{A(\one)/B(\one)}}{2\sqrt{A(\one)B(\one)} + Q(\one)}\right)\] 
along with (potentially) other critical points with the same coordinate-wise modulus and non-minimal critical points that do not affect dominant asymptotics. In addition, the intersection $\sing_1 \cap \sing_2$ is a manifold containing the critical point
\[ \bsigma = \left(\one,1,S(\one)\right)\] 
and (once again) potentially other critical points with the same coordinate-wise modulus. The points near which analytic behaviour of $F$ determines dominant asymptotics of $s_n$ depends on the drift of the model.

\subsubsection{Negative Drift Models}

When $\mS$ has negative drift then $A(\one)>B(\one)$ and the point $\brho$ is a smooth minimal critical point of $\sing$. 

\begin{example}
Consider the (unweighted) mostly symmetric model with step set $\mS = \{(-1,-1),(1,-1),(0,1)\}$
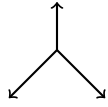
\[ \diagr{N,SE,SW} \]
whose generating function is the diagonal of
\[ F(x,y,t) = \frac{(1 + x)(1-2t(x^2y^2+1))}{(1 - y)(1- t(x^2y^2 + y^2 +x))(1-t(x^2y^2+1))}. \]
The critical point $\brho = (1,1/\sqrt{2},1/2)$ is minimal, as are the minimal critical points $(1,-1/\sqrt{2},1/2)$ and $(-1,\pm i/\sqrt{2},-1/2)$, and this time both $\brho$ and $(1,-1/\sqrt{2},1/2)$ contribute to dominant asymptotics of $s_n$. Because $H_2$ is the only denominator factor vanishing at these points, the smooth analysis used in Example~\ref{ex:smooth} above still applies. Computing a residue and applying the saddle-point method to determine the asymptotic contribution of each of these points ultimately yields
\[ s_n \sim \frac{16+12\sqrt{2}}{\pi} \cdot \frac{(2\sqrt{2})^n}{n^2} + \frac{-16+12\sqrt{2}}{\pi} \cdot \frac{(-2\sqrt{2})^n}{n^2} 
= \begin{cases} \frac{24\sqrt{2}}{\pi} \cdot \frac{(2\sqrt{2})^n}{n^2} &: n \text{ even} \\[+2mm]
\frac{32}{\pi} \cdot \frac{(2\sqrt{2})^n}{n^2} &: n \text{ odd} \end{cases}. \]
\end{example}

\subsubsection{Positive Drift Models}

When $\mS$ has positive drift then $A(\one)<B(\one)$ and the point $\btau$ is no longer minimal, as its $d$th coordinate is larger than 1. In fact, the only minimal critical point is the non-smooth point $\bsigma$ (and potentially other points with the same coordinate-wise modulus that do not affect dominant asymptotics) and a more general argument must be applied.

\begin{example}
\label{ex:pos}
Consider the (unweighted) mostly symmetric model with step set $\mS = \{(-1,1),(1,1),(0,-1)\}$
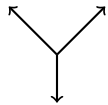
\[ \diagr{S,NE,NW} \]
whose generating function is the diagonal of
\[ F(x,y,t) = \frac{(1 + x)(1-2txy^2)}{(1 - y)(1- t(xy^2 + x^2 +1))(1-txy^2)}. \]
The critical points defined by the vanishing of $H_2(x,y,t) = 1- t(xy^2 + x^2 +1)$ have coordinate-wise modulus $(1,\sqrt{2},1/4)$ and are not minimal as expected. Because $\bsigma=(1,1,1/3)$ is the only minimal critical point, we start by writing
\begin{equation} 
s_n = \frac{1}{2\pi i} \int_{|x|=1} \left(\frac{1}{(2\pi i)^2}\int_{T(1-\epsilon_1,1/3-\epsilon_2)} \frac{P(x,y)}{(1 - y)(1- t(xy^2 + x^2 +1))} \frac{dydt}{(xyt)^{n+1}} \right)  \, dx
\label{eq:pos}
\end{equation}
where we use the notation $T(\ww) = \{\zz :|z_j|=|w_j| \text{ for all } j\}$, the constants $\epsilon_1$ and $\epsilon_2$ are sufficiently small, and 
\[ P(x,y,t) = \frac{(1 + x)(1-2txy^2)}{(1-txy^2)}\] 
captures the numerator and the denominator factor that does not vanish near $\bsigma$. If $\mN(\ww)$ denotes a subset of $T(\ww)$ where each coordinate is sufficiently close to the positive real axis then an analytic argument shows that we can restrict $x$ to $\mN(1)$ and $(y,t)$ to $\mN(1-\epsilon_1,1/3-\epsilon_2)$.

Our next goal is to compute residues to remove two variables, one for each vanishing denominator factor. To that end, we note that the part of the integrand in~\eqref{eq:pos} that depends on $n$ is captured by the function $\phi(x,y,t) = (xyt)^{-1}$. Since
\[ \underbrace{(\grad \phi)(1,1,1/3)}_{(-3,-3,-9)} = \underbrace{(\grad H_1)(1,1,1/3)}_{(0,-1,0)} \; + \; 3 \; \underbrace{(\grad H_2)(1,1,1/3)}_{(-1,-2/3,-3)} \]
is a linear combination with positive coefficients, replacing the domain $\mN(1-\epsilon_1,1/3-\epsilon_2)$ in~\eqref{eq:pos} with the domains $\mN(1+\epsilon_1,1/3-\epsilon_2), \mN(1-\epsilon_1,1/3+\epsilon_2),$ and $\mN(1+\epsilon_1,1/3+\epsilon_2)$ that `cross' the singular sets $\mV_1$ and $\mV_2$ results in integrals that are exponentially smaller than the one under consideration. A signed sum of these integrals is computed by taking a `multivariate residue' at the points where $y=1$ and $t=1/(xy^2 + x^2 +1)=1/(x^2+x+1)$, ultimately giving that
\begin{align*} 
s_n &\sim \frac{1}{2\pi i} \int_{\mN(1)} \frac{P(x,1)}{x} (x + 1 + 1/x)^n \, dx \\[+2mm]
&= \frac{1}{2\pi i} \int_{\mN(1)} \frac{(x^2 - x + 1)(1 + x)}{x(x^2 + 1)} (x + 1 + 1/x)^n \, dx.
\end{align*}

Making the change of variables $x=e^{i\theta}$ for $\theta$ in a neighbourhood $\mM$ of the origin, then applying the saddle-point method, gives 
\[ s_n \sim \frac{1}{2\pi} \int_{\mM} \frac{(e^{2i\theta} - e^{i\theta} + 1)(1 + e^{i\theta})}{(e^{2i\theta} + 1)} e^{n\log(e^{i\theta} + 1 + e^{-i\theta})} \, d\theta \sim \frac{1}{2\pi} \int_{\R} e^{n\log(3)-n\theta^2/3} d\theta = \frac{\sqrt{3}}{2\sqrt{\pi}} \cdot \frac{3^n}{\sqrt{n}}. \]
\end{example}

In the mostly symmetric negative drift case the gradient of $\phi$ at $\bsigma$ cannot be written as a positive linear combination of the gradients of $H_1$ and $H_2$ at $\bsigma$, so the non-smooth minimal critical point $\bsigma$ cannot be used to characterize asymptotics. Of course, for negative drift models the smooth critical point $\brho$ is minimal and can be used instead. The fact that smooth minimal critical points arise precisely when non-smooth critical points cannot be used to compute asymptotics holds more generally, and is explained by the theory of ACSV for \emph{multiple points} (see~\cite[Chapter 9]{Melczer2021} or~\cite[Chapter 10]{PemantleWilsonMelczer2024}). 

\subsubsection{Zero Drift Models}

In the zero drift case $A(\one)=B(\one)$, so the critical points $\brho$ and $\bsigma$ collide. This makes two denominator factors vanish at the minimal critical point determining asymptotics, but unlike the positive drift case we can only take one residue. Computing asymptotics thus requires a more careful analysis, conducted in Section~\ref{sec:mainProof} to prove Theorem~\ref{thm:main}.

\begin{example}
\label{ex:zerodrift}
Consider the mostly symmetric zero drift model with step set $\mS = \{(0,1),(0,1),(-1,-1),(1,1)\}$ whose North step appears twice (i.e., has a weight of two).
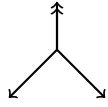
\[ \diagr{N,N2,SE,SW} \]
The generating function for this walk model can be written as the diagonal of 
\[ F(x,y,t) = \frac{(1 + x)(1 - 2ty^2(x^2 + 1))}{(1 - y)(1 - txy(2/y + xy + y/x))(1 - ty^2(x^2 + 1))}, \]
with minimal critical points $\brho=\bsigma=(1,1,1/4)$ with $(1,-1,-1/4)$ and $(-1,\pm i,\mp i)$. Because $H_2$ is the only denominator factor of $H$ that vanishes at the final three points, the contributions of these points to the asymptotic behaviour of $s_n$ are covered by smooth ACSV: the fact that the numerator vanishes at each of these points further implies that these contributions are bounded in $O(4^n n^{-2})$, which will turn out not to affect dominant asymptotics of $s_n$.

Analogously to the positive drift case, we begin by writing
\begin{equation} 
s_n = \frac{1}{2\pi i} \int_{T(1)} \left(\frac{1}{(2\pi i)^2}\int_{T(1-\epsilon_1,1/4-\epsilon_2)} \frac{P(x,y)}{(1 - y)(1 - txy(2/y + xy + y/x))} \frac{dydt}{(xyt)^{n+1}} \right)  \, dx
\label{eq:zer}
\end{equation}
where the constants $\epsilon_1$ and $\epsilon_2$ are sufficiently small, and 
\[ P(x,y,t) = \frac{(1 + x)(1 - 2ty^2(x^2 + 1))}{1 - ty^2(x^2 + 1)}\] 
captures the numerator and the denominator factor that does not vanish near $\bsigma$. Again we can replace the domains of integration $T(1)$ and $T(1-\epsilon_1,1/4-\epsilon_2)$ by subsets $\mN(1)$ and $\mN(1-\epsilon_1,1/4-\epsilon_2)$ near the positive real axes while introducing a negligible error. The part of the integrand in~\eqref{eq:zer} that depends on $n$ is still captured by the function $\phi(x,y,t) = (xyt)^{-1}$, but now
\[ \underbrace{(\grad \phi)(1,1,1/4)}_{(-4,-4,-16)} = 0 \; \underbrace{(\grad H_1)(1,1,1/4)}_{(0,-1,0)} \; + \; 4 \; \underbrace{(\grad H_2)(1,1,1/4)}_{(-1,-1,-4)} \]
is a linear combination with one coefficient equal to zero. Because of this, we can replace $\mN(1-\epsilon_1,1/4-\epsilon_2)$ by a domain $\mN(1-\epsilon_1,1/4+\epsilon_2)$ `crossing' the factor $H_2$ and get something growing exponentially smaller than~$s_n$, but (unlike the positive drift case) we can no longer introduce the integrals whose domains of integration `cross' $H_1$. Performing the residue computation for the pole $t=1/xy\oS(x,y)$ gives
\[ s_n \sim \frac{1}{(2 \pi i)^2} \int_{\mN(1)} \int_{\mN(1-\epsilon_1)} \frac{1+x}{2x^2y} \cdot \frac{2x - y^2(1+x^2)}{1-y} \oS(x,y)^n dx dy , \]
and the $1-y$ factor in the denominator complicates the analysis. We note, however, that the numerator factor $2x - y^2(1+x^2)$ also vanishes at $x=y=1$ and 
\[ \frac{2x - y^2(1+x^2)}{1-y} = 4 + O(1-x) + O(1-y) + O\left(\frac{(1-x)^2}{1-y}\right). \]
The difficult factor in the denominator can thus be handled by carefully taking $\epsilon_1\rightarrow0$ at a rate that balances the integral manipulations and approximations above with the need to approximate the integrand by its leading series terms near $(1,1)$, ultimately giving
\[ s_n \sim \frac{4}{(2 \pi i)^2} \int_{\mN(1,1)} \frac{1+x}{2x^2y} \cdot \oS(x,y)^n dx dy . \]
Making the change of variables $x=e^{i\theta_1}$ and $y=e^{i\theta_2}$ and applying the saddle-point method then implies
\[ s_n \sim \frac{4}{(2 \pi)^2} \int_{\R^2} e^{n \log 4 - n(\theta_1^2/4 + \theta_2^2/2)} d\theta_1 d\theta_2 = \frac{4^n}{n} \cdot \frac{2\sqrt{2}}{\pi}. \]
\end{example}

In Section~\ref{sec:mainProof} we formalize the argument sketched in Example~\ref{ex:zerodrift} to prove Theorem~\ref{thm:main}. Computing higher order terms in the asymptotic expansion of $s_n$, and dealing with the general \emph{non-generic} case where less residues can be taken than expected, requires approximating saddle-point-like integrals with amplitudes that are non-singular at their saddle points. A method handling the general case will be detailed in a forthcoming paper, and we illustrate the approach in Example~\ref{ex:zerodrift2} of Section~\ref{sec:conclusion} to show that
\[ s_n = \frac{4^n}{n} \left(\frac{2\sqrt{2}}{\pi} + \frac{1}{\sqrt{\pi}}n^{-1/2} + O(n^{-1})\right) \]
for the lattice path model in Example~\ref{ex:zerodrift}.

\section{Asymptotics of Zero Drift Mostly Symmetric Models}
\label{sec:mainProof}
To prove Theorem~\ref{thm:main} we now fix a mostly symmetric zero drift step set $\mS$. The diagonal representation~\eqref{eq:mostlyDiag}, combined with the Cauchy integral formula, implies that
\[ s_n = \frac{1}{(2 \pi i)^{d+1}} \int_{T \times C^{\textsl{in}}} \frac{ G(\zz,t) }{H(\zz,t) } \frac{d\zz d t}{(z_1 \cdots z_d \, t)^{n+1}} \]
where
\begin{align*} 
T &= T_\epsilon= \{\zz\in\C^d : |z_k|=1 \text{ for all } 1 \leq k \leq d-1 \text{ and } |z_d|=1-\epsilon\} \\[+2mm]
C^{\textsl{in}} &= C^{\textsl{in}}_\gamma = \{t\in\C : |t| = S(\one)^{-1}(1-\gamma) \}
\end{align*}
for any $0<\epsilon,\gamma<1$, the numerator $G$ satisfies
\[ G(\bz,t) = (1+z_1) \cdots (1+z_{d-1}) \left( 1- tz_1 \cdots z_d \left(Q(\htbz) + 2 z_d A(\htbz) \right) \right) ,\]
and the denominator $H$ can be written $H=H_1H_2H_3$ for
\begin{align*}
H_1(\zz,t) &= 1-z_d \\
H_2(\zz,t) &= 1- tz_1 \cdots z_d \oS(\zz) \\
H_3(\zz,t) &= 1- tz_1 \cdots z_d \left(Q(\htbz) +  z_d A(\htbz) \right) .
\end{align*}
We make this choice of $T$ in order to stay in the domain of convergence of the series expansion under consideration, and because -- as shown in Melczer and Wilson~\cite[Proposition 4.2]{MelczerWilson2019} -- the point $(\one,S(\one))$ is the unique minimal zero of $H(\zz,t)$ with positive coordinates that minimizes $|z_1 \cdots z_d t|^{-1}$. It will also be of importance that the zeroes of $H_2$ with the same coordinate wise-modulus as $(\one,S(\one))$ are given by $(\ww,t)$ where $t = 1/(w_1\cdots w_d S(\ww,w_d))$ and $\ww$ is contained in the set 
\begin{align}
\label{eq:critical-point-set}
    \Gamma = \left\{(\htbw,w_d) : \htbw \in \{\pm1\}^{d-1}, \, w_d \in\{\pm1,\pm i\}, \quad \left|\frac{1}{w_1\cdots w_d S(\htbw,w_d)}\right| = \frac{1}{S(\one)} \right\};
\end{align}
see~\cite[Theorem 4.3]{MelczerWilson2019} (although we note that those authors forgot to allow the possibility that $w_d=\pm i$ in their argument).

\subsection{Residue Computations}

We use residue computations to simplify the integral under consideration. First, we prove that we can restrict our integral to neighbourhoods of points in $\Gamma$ and get an exponentially small error. Given $\ww \in \Gamma$, we define the product of circle arcs contained in $T$
\[ \mN_\ww = \mN^{\delta,\epsilon}_\ww = \{\zz\in\C^d:|z_j|=1 \text{ for $1\leq j\leq d-1$ with $|z_d| = 1-\epsilon$ and } |\arg(z_j)-\arg(w_j)|<\delta \text{ for all } j\},\] 
which will shrink to $\ww$ as $n\rightarrow\infty$, and let
\[ \mN = \bigcup_{\ww \in \Gamma}\mN_\ww. \]

\paragraph{Running assumptions on $\epsilon$ and $\delta$}
The asymptotic arguments below yield dominant asymptotics of $s_n$ when $\epsilon=n^{-\alpha}$ and $\delta=n^{-\beta}$ for positive constants $\alpha$ and $\beta$ such that 
\[ 1/2<\alpha<2\beta, \hspace{0.2in} \alpha+\beta>1, \hspace{0.1in}\text{and}\hspace{0.1in} 1/3 < \beta < 1/2,\]
and we now assume these inequalities hold. For concreteness, one can take $\epsilon = n^{-7/10}$ and $\delta = n^{-2/5}$.

\begin{lemma}
\label{lem:localize}
Under our running assumptions on $\epsilon$ and $\delta$, if $0<\gamma<1$ then
\[ s_n = \frac{1}{(2 \pi i)^{d+1}} \int_{\mN \times C^{\textsl{in}}} \frac{ G(\zz,t) }{H(\zz,t) } \frac{d\zz d t}{(z_1 \cdots z_d \, t)^{n+1}} + O(S(\one)^ne_n) \]
where $e_n\rightarrow0$ faster than any fixed power of $n$.
\end{lemma}

\begin{proof}
By our Cauchy integral representation for $s_n$, it suffices to bound
\begin{equation*} 
\frac{1}{(2 \pi i)^{d+1}} \int_{(T\setminus \mN) \times C^{\textsl{in}}} \frac{ G(\zz,t) }{H(\zz,t) } \frac{d\zz d t}{(z_1 \cdots z_d \, t)^{n+1}} 
\end{equation*}
for suitable values of $\epsilon,\delta,$ and $\gamma$. First, we note that if $\zz \in T$ is fixed then $C^{\textsl{in}}$ is contained in the power series domain of convergence of $G(\zz,t)/H(\zz,t)$, so
\begin{equation*}
    \frac{1}{2 \pi i} \int_{(T\setminus \mN) \times C^{\textsl{in}}} \frac{ G(\zz,t) }{H(\zz,t) } \frac{d\zz d t}{(z_1 \cdots z_d \, t)^{n+1}} = \int_{(T\setminus \mN)} \left[ \frac{1}{2 \pi i}\int_{C^{\textsl{in}}} \frac{G(\zz,t)}{H(\zz,t)} \frac{ d t}{t^{n+1}} \right] \frac{d\zz }{(z_1 \cdots z_d )^{n+1}},
\end{equation*}
where
\[ \frac{1}{2 \pi i}\int_{C^{\textsl{in}}} \frac{ G(\zz,t) }{H(\zz,t) } \frac{ d t}{t^{n+1}} = [t^n] F(\zz,t) \]
is the coefficient of $t^n$ when $F(\zz,t)=G(\zz,t)/H(\zz,t)$ is expanded as a power series. When each coordinate $z_j$ has unit modulus and $\zz$ is bounded away from the elements of $\Gamma$ then $Q(\htbz)+z_d A(\htbz)$ and $\oS(\zz)$ are less than $S(\one)$ (see~\cite[Proposition 4.2]{MelczerWilson2019}) so that $[t^n] F(\zz,t)= O(\tau^n)$ for some $0<\tau<S(\one)$. In particular, \emph{if $\delta$ was a fixed constant} then 
\begin{align*} 
\left| \frac{1}{(2 \pi i)^{d+1}} \int_{(T\setminus \mN) \times C^{\textsl{in}}} \frac{ G(\zz,t) }{H(\zz,t) } \frac{d\zz d t}{(z_1 \cdots z_d \, t)^{n+1}}\right|
&\leq  \frac{1}{(2 \pi)^{d}} \int_{(T\setminus \mN)} \left|[t^n] F(\zz,t)\right| \cdot \left|z_1 \cdots z_{d-1}z_d\right|^{-n-1} d\zz 
\end{align*}
grows exponentially slower than $S(\one)^n$ as $|z_j|=1$ for all $1 \leq j \leq d-1$ and $|z_d|=1-n^{-\alpha}$ for $\zz\in T$.

In actuality $\delta=n^{-\beta}\rightarrow0$ as $n\rightarrow\infty$ and we need to show that $\delta$ does not go to zero too quickly. By the definition of $\Gamma$, the monomials of $\oS(\zz)$ all have the same argument at $\zz=\ww$ for $\ww \in \Gamma$, so 
\[ |Q(\htbz)+z_d A(\htbz)| < |\oz_d B(\htbz) + Q(\htbz)+z_d A(\htbz)| = |\oS(\zz)| \]
for $\zz$ sufficiently close to the elements of $\Gamma$. Because we can reduce to any arbitrarily small neighbourhoods of the elements of $\Gamma$ while introducing only an exponentially smaller error, we can thus assume that we are in a small enough neighbourhood such that $\left|[t^n] F(\zz,t)\right| = O(|\oS(\zz)|^n)$. We bound $|\oS(\zz)|^n$ for $\zz$ near $\one \in \Gamma$, with the argument for the other elements $\ww\in\Gamma$ being analogous. Writing $z_j=e^{i\theta_j}$ for $1 \leq j \leq d-1$ and $z_d = (1-\epsilon)e^{i\theta}$ we have for $\theta_j$ sufficiently small that
\begin{align*}
\oS(\zz)^n 
&= \exp{n \log \oS\left(e^{i\hat{\btheta}},(1-\epsilon)e^{i\theta_d})\right)} \\
&= S(\one)^n \cdot \exp{n \log \oS\left(e^{i\btheta}\right) - n \log S(\one) + O(n\epsilon\theta_1) + \cdots + O(n\epsilon\theta_{d-1}) + O(n\epsilon^2)} \\[+2mm]
&= S(\one)^n \cdot \exp{-n(c_1\theta_1^2 + \cdots + c_d\theta_d^2) + O(n(\theta_1+\cdots + \theta_d)^3) + O(n\epsilon\theta_1) + \cdots + O(n\epsilon\theta_{d-1}) + O(n\epsilon^2)},
\end{align*}
where the second equality follows from the expansion of $\log \oS(e^{i\hat{\btheta}},(1-\epsilon)e^{i\theta_d})$ derived in the proof of Proposition~\ref{prop:leadingsaddle} below, and the $c_j$ are positive constants computed in the proof of Proposition~\ref{prop:leadingasm} below. When $|\theta_j|\geq\delta$ then, under our assumptions on $\epsilon=n^{-\alpha}$ and $\delta=n^{-\beta}$ and the fact that we can restrict the $\theta_j$ to any arbitrarily small neighbourhood of the origin, the big-O terms in the exponential are bounded below the quadratic term so 
\[ \left|[t^n] F(\zz,t)\right| = O(|\oS(\zz)|^n) = O(S(\one)^n e^{-Kn^{\rho}}) \]
for some $K,\rho>0$. The claimed result holds as $e^{-Kn^{\rho}}\rightarrow0$ faster than any fixed power of~$n$. 
\end{proof}

Having localized near particular singularities of interest, we now introduce an asymptotically negligible integral that allows us to make our desired residue computation. Let 
\[ C^{\textsl{out}} = C^{\textsl{out}}_\gamma = \{t\in\C : |t| = S(\one)^{-1}(1+\gamma) \}. \]

\begin{lemma}
\label{lem:outer}
Using the notation and running assumptions above, if $\gamma>0$ is sufficiently small then
\[  \frac{1}{(2 \pi i)^{d+1}} \int_{\mN \times C^{\textsl{out}}} \frac{ G(\zz,t) }{H(\zz,t) } \frac{d\zz d t}{(z_1 \cdots z_d \, t)^{n+1}} = O(\tau^n) \]
for some $0<\tau<S(\one)$. 
\end{lemma}

\begin{proof}
Bounding the integrand under consideration by the supremum of its modulus, there exists a constant $K_1>0$ such that
\begin{equation} 
\left| \int_{\mN \times C^{\textsl{out}}} \frac{ G(\zz,t) }{H(\zz,t) } \frac{d\zz d t}{(z_1 \cdots z_d \, t)^{n+1}}\right| 
\leq K_1 \sup_{(\zz,t) \in \mN \times C^{\textsl{out}}} \left| \frac{ G(\zz,t) }{H(\zz,t) } \right| \left[ \frac{S(\one)}{(1-\epsilon) (1 + \gamma)} \right]^n.
\label{eq:outer}
\end{equation}
When $\epsilon$ and $\delta$ are sufficiently small then, by the definition of $\Gamma$, the only poles of $F(\zz,t) = G(\zz,t)/H(\zz,t)$ with $\zz \in \mN$ have $|t|$ sufficiently close to $S(\one)^{-1}$ or $(Q(\one) + A(\one))^{-1}$. In particular, if $\gamma$ is fixed sufficiently small and~$\epsilon$ and $\delta$ are sufficiently smaller than $\gamma$ then there exists $K_2>0$ such that
\[ \sup_{(\zz,t) \in \mN \times C^{\textsl{out}}} \left| \frac{ G(\zz,t) }{H(\zz,t) } \right| \leq K_2 \sup_{|z_d|=1-\epsilon} \left|\frac{1}{1-z_d}\right| = O(n^{\alpha}), \]
since $\epsilon=n^{-\alpha}$. The right-hand side of~\eqref{eq:outer} thus has exponential growth smaller than $S(\one)$ whenever $(1- \epsilon) (1 + \gamma) >1$, which occurs for all sufficiently large $n$ under our assumptions.
\end{proof}

Our last result before computing residues shows that the only poles of our integrand that affect our calculations come from the denominator factor $H_2$ vanishing.

\begin{lemma}
\label{lem:otherfactors}
The polynomial $H_3(\zz,t) = 1- tz_1 \cdots z_d \left(Q(\htbz) +  z_d A(\htbz) \right)$ is non-zero when $|z_j| \leq 1$ for all $1 \leq j \leq d$ and $|t| \leq S(\one)$. If $\delta$ and $\epsilon$ are sufficiently small then $B(\htbz)\neq0$ for $\htbz \in \mN_\ww$.
\end{lemma}

\begin{proof}
Note that $\oS(\bz) = S(z_1,\dots,z_{d-1},\oz_d) = z_d A\left(\htbz\right) + Q\left(\htbz\right) + \oz_d B\left(\htbz\right)$ so
\[ \left|Q(\htbz) +  z_d A(\htbz)\right| \leq Q(\one) + A(\one) < S(\one) \]
when  $|z_j| \leq 1$ for all $1 \leq j \leq d$. In particular, if $|z_j| \leq 1$ for all $1 \leq j \leq d$ and $H_3(\zz,t)=0$ then
\[ |t| = \left|\frac{1}{z_1 \cdots z_d \left(Q(\htbz) +  z_d A(\htbz) \right)}\right| > \frac{1}{S(\one)}. \]

To prove that $B(\htbz)\neq0$ for $\htbz \in \mN_\ww$ we note that $B(\htbz)$ does not vanish whenever $\htbz$ is a maximizer of~$|\oS(\htbz)|$ on the unit torus. The points of $\Gamma$ get arbitrarily close to these maximizers as $\epsilon\rightarrow0$, so by continuity $B(\htbz)$ does not vanish for $\htbz \in \mN_\ww$ whenever $\delta$ and $\epsilon$ are sufficiently small.
\end{proof}

We are now ready to determine a residue integral expression for $s_n$.

\begin{proposition}
\label{prop:resint}
With the notations and assumptions above,
\begin{equation}
s_n = \sum_{\ww\in\Gamma} \frac{1}{(2 \pi i)^d} \int_{\mN_\ww} \frac{(1+z_1)\cdots(1+z_{d-1})}{B\left(\htbz\right)(z_1\cdots z_d)} \cdot \frac{B\left(\htbz\right) - z_d^2 A\left(\htbz\right)}{1-z_d} \cdot \oS(\zz)^n \, d\zz + O((S(\one)^ne_n)
\label{eq:resInt}
\end{equation}
where $e_n\rightarrow0$ faster than any fixed power of $n$.
\end{proposition}

\begin{proof}
Lemmas~\ref{lem:localize} and~\ref{lem:outer} imply that we can write
\[ s_n = \frac{-1}{(2 \pi i)^{d}} \int_{\mN} \left(\frac{1}{2\pi i} \int_{C^{\textsl{out}}} \frac{ G(\zz,t) }{H(\zz,t) } \frac{ d t}{(z_1 \cdots z_d \, t)^{n+1}} - \frac{1}{2\pi i} \int_{C^{\textsl{in}}} \frac{ G(\zz,t) }{H(\zz,t) } \frac{ d t}{(z_1 \cdots z_d \, t)^{n+1}} \right) d\zz + O((S(\one)^ne_n) \]
under our assumptions. The first statement of Lemma~\ref{lem:otherfactors} implies that the inner difference of integrals is given by a residue where $t = 1/\left(z_1 \cdots z_d \oS(\zz)\right)$, and the second statement of Lemma~\ref{lem:otherfactors} implies that the poles involved are all simple. Computing the residue gives
\[ 
-\frac{(1+z_1)\cdots(1+z_{d-1})\left(1 - \frac{Q\left(\htbz\right)+2z_d A\left(\htbz\right)}{\oz_d B\left(\htbz\right) + Q\left(\htbz\right) + z_d A\left(\htbz\right)} \right)}
{(z_1\cdots z_d \oS(\zz))(1-z_d) \left(1 - \frac{Q\left(\htbz\right)+z_d A\left(\htbz\right)}{\oz_d B\left(\htbz\right) + Q\left(\htbz\right) + z_d A\left(\htbz\right)}\right)} = -\frac{(1+z_d)\cdots(1+z_{d-1})\left(1 - z_d^2\frac{A\left(\htbz\right)}{B\left(\htbz\right)}\right)}{(1-z_d)(z_1\cdots z_d \oS(\zz))}.
\]
\end{proof}

\subsection{Computing Dominant Asymptotics}
\label{subseq:asymptotics}
To determine asymptotics of $s_n$ it is sufficient to determine asymptotics for the residue integrals in~\eqref{eq:resInt}. If the factor of $1-z_d$ could be cancelled from the numerator (which occurs in the highly symmetric case, when $A(\htbz)=B(\htbz)$) then asymptotically evaluating these integrals is a straightforward application of the saddle-point method. The presence of this factor complicates the analysis, which is why the non-highly symmetric zero drift case was left open in~\cite{MelczerWilson2019}. To determine asymptotics, we thus switch to polar coordinates and carefully consider the resulting saddle-point-like integrals.

We start by reducing the contribution of the point $\bw=\one$, which turns out to be the only point contributing to dominant asymptotics, to a standard saddle-point integral.

\begin{proposition}
\label{prop:leadingsaddle}
Under our running assumptions on $\epsilon$ and $\delta$,
\[
\frac{1}{(2\pi i)^d}\int_{\mN_\one} \frac{(1+z_1)\cdots(1+z_{d-1})}{B\left(\htbz\right)(z_1\cdots z_d)} \cdot \frac{B\left(\htbz\right) - z_d^2 A\left(\htbz\right)}{1-z_d} \cdot \oS(\zz)^n \, d\zz 
\sim \frac{1}{\pi^d} \int_{\zz \in \mM} e^{n  \log \oS\left(e^{i\btheta}\right)} d\btheta
\]
where $\mM = (-\delta,\delta)^d$ and $e^{i\btheta}=\left(e^{i\theta_1},\dots,e^{i\theta_{d}}\right)$.
\end{proposition}

\begin{proof}
Making the change of variables $z_j = e^{i\theta_j}$ for $1 \leq j \leq d-1$ and $z_d = (1-\epsilon)e^{i\theta_d}$ gives
\[ \frac{1}{(2\pi i)^d}\int_{\mN_\one} \frac{(1+z_1)\cdots(1+z_{d-1})}{B\left(\htbz\right)(z_1\cdots z_d)} \cdot \frac{B\left(\htbz\right) - z_d^2 A\left(\htbz\right)}{1-z_d} \cdot \oS(\zz)^n \, d\zz 
= \frac{1}{(2\pi)^d}\int_{\zz \in \mM} P(\btheta)Q(\btheta) e^{n \phi(\btheta)} d\btheta,  \]
where
\begin{align*}
P(\btheta) &= \frac{(1+e^{i\theta_1})\cdots(1+e^{i\theta_{d-1}})}{B\left(e^{i\hat{\btheta}}\right)} \\[+2mm]
Q(\btheta) &= \frac{B\left(e^{i\hat{\btheta}}\right) - (1-\epsilon)^2e^{2i\theta_d} A\left(e^{i\hat{\btheta}}\right)}{1-(1-\epsilon)e^{i\theta_d}} \\[+2mm]
\phi(\btheta) &= \log \oS\left(e^{i\hat{\btheta}},(1-\epsilon)e^{i\theta_d}\right)
\end{align*}
and $e^{i\hat{\btheta}}=\left(e^{i\theta_1},\dots,e^{i\theta_{d-1}}\right)$. The first factor is easiest to analyze as $\epsilon,\delta\rightarrow0$, with a power series expansion of the numerator and denominator showing that
\[ P(\btheta) = \frac{2^{d-1}}{B(\one)} + O(\epsilon + \theta_1 + \cdots + \theta_{d-1}). \]
Since $A(\htbz)$ and $B(\htbz)$ are invariant under the transformations $z_j=1/z_j$ for $1 \leq j \leq d-1$, the functions $A\left(e^{i\hat{\btheta}}\right)$ and $B\left(e^{i\hat{\btheta}}\right)$ are even. Combined with the fact that $A(\one)=B(\one)$, we see that
\begin{align*} 
Q(\btheta) 
&= \frac{B(\one)-(1-\epsilon)^2e^{2i\theta_d} A(\one) + O(\theta_1^2 + \cdots + \theta_{d-1}^2)}{1-(1-\epsilon)e^{i\theta_d}} \\[+2mm]
&= B(\one)(1+(1-\epsilon)e^{i\theta_d}) + O\left(\frac{\theta_1^2 + \cdots + \theta_{d-1}^2}{1-(1-\epsilon)e^{i\theta_d}}\right) \\[+2mm]
&= 2B(\one) + O(\epsilon + \theta_d + \theta_1^2\epsilon^{-1} + \cdots + \theta_{d-1}^2\epsilon^{-1}). 
\end{align*}
Finally, we have
\begin{align*} 
\phi(\btheta) 
&= \log \left((1-\epsilon)e^{i\theta_d}A\left(e^{i\hat{\btheta}}\right) + Q\left(e^{i\hat{\btheta}}\right) + \frac{B\left(e^{i\hat{\btheta}}\right)}{(1-\epsilon)e^{i\theta_d}} \right) \\[+2mm]
&=\log \left( e^{i\theta_d}A\left(e^{i\hat{\btheta}}\right) + Q\left(e^{i\hat{\btheta}}\right) + e^{-i\theta_d}B\left(e^{i\hat{\btheta}}\right)\right) 
+ \epsilon \frac{B\left(e^{i\hat{\btheta}}\right) - e^{i\theta_d}A\left(e^{i\hat{\btheta}}\right)}{e^{i\theta_d}A\left(e^{i\hat{\btheta}}\right) + Q\left(e^{i\hat{\btheta}}\right) + e^{-i\theta_d}B\left(e^{i\hat{\btheta}}\right)} 
+ O(\epsilon^2) \\[+2mm]
&=\log \left( e^{i\theta_d}A\left(e^{i\hat{\btheta}}\right) + Q\left(e^{i\hat{\btheta}}\right) + e^{-i\theta_d}B\left(e^{i\hat{\btheta}}\right)\right) 
+ \epsilon \left(\frac{B\left(e^{i\hat{\btheta}}\right) - A\left(e^{i\hat{\btheta}}\right)}{A\left(e^{i\hat{\btheta}}\right)+Q\left(e^{i\hat{\btheta}}\right)+B\left(e^{i\hat{\btheta}}\right)}+O(\theta_d)\right) 
+ O(\epsilon^2) \\[+2mm]
&= \log \oS\left(e^{i\btheta}\right) + O(\epsilon\theta_1 + \cdots + \epsilon\theta_d + \epsilon^2).
\end{align*}
Putting everything together gives
\[ \frac{1}{(2\pi)^d}\int_{\zz \in \mM} P(\btheta)Q(\btheta) e^{n \phi(\btheta)} d\btheta \sim \frac{1}{\pi^d} \int_{\zz \in \mM} e^{n  \log \oS\left(e^{i\btheta}\right) } d\btheta,
\]
since each $|\theta_j| <\delta$ on $\mM$ and all of $\epsilon,\delta,\delta^2\epsilon^{-1},n\epsilon\delta,n\epsilon^2 \rightarrow 0$ go to zero as $n\rightarrow\infty$ under our assumptions on $\epsilon$ and $\delta$.
\end{proof}

Having removed the dependence on $\epsilon$ for the integral determining dominant asymptotics, we now determine asymptotic behaviour using a saddle-point analysis.

\begin{proposition}
\label{prop:leadingasm}
Under our running assumptions on $\epsilon$ and $\delta$,
\[
\frac{1}{\pi^d} \int_{\zz \in \mM} e^{n  \log \oS\left(e^{i\btheta}\right)} d\btheta
\sim 
S(\bone)^n n^{-d/2} \frac{S(\bone)^{d/2}}{ \pi^{d/2} (b_1 \cdots b_d)^{1/2}},
\]
where we recall from above that $b_k = \sum_{\bi \in \mS, i_k=1} w_{\bi}$ is the total weight of the steps in $\mS$ moving forward in the $k$th coordinate.
\end{proposition}

\begin{proof}
We claim that there is an expansion
\begin{equation}
\log \oS\left(e^{i\btheta}\right) = \log S(\bone) - \sum_{j=1}^{d}\frac{ b_j }{S(\bone)}\theta_j^2 + O ( (\theta_1 + \cdots + \theta_d)^3 ).
\label{eq:logphiexpand}
\end{equation}
Assuming this claim, for $n$ sufficiently large the domain $\mM$ is a small enough neighbourhood of the origin that we can write
\begin{align*}
    \int_{\mM} e^{n \log \oS\left(e^{i\btheta}\right) } d\btheta &= S(\bone)^n \int_{\zz \in \mM} e^{ - n \sum_{j=1}^{d}\frac{ b_j }{S(\bone)}\theta_j^2 + O ( n(\theta_1 + \cdots + \theta_d)^3 ) } d\btheta\\
    &= S(\bone)^n \left(\int_{\zz \in \mM} e^{ - n \sum_{j=1}^{d}\frac{ b_j }{S(\bone)}\theta_j^2 } d\btheta \right) \big(1 + O(n\delta^3) \big)\\
    &\sim S(\bone)^n \int_{\zz \in \mM} e^{ - n \sum_{j=1}^{d}\frac{ b_j }{S(\bone)}\theta_j^2} d \btheta\\
    &= S(\bone)^n \prod_{j = 1}^d 
    \int_{-\delta}^\delta e^{-n \frac{b_j}{S(\bone)} \theta_j^2 } d \theta_j\\
    &\sim S(\bone)^n \prod_{j = 1}^d 
    \int_{-\infty}^\infty e^{-n \frac{b_j}{S(\bone)} \theta_j^2 } d \theta_j\\
    &= S(\bone)^n n^{-d/2} \frac{\pi^{d/2} S(\bone)^{d/2}}{ (b_1 \cdots b_d)^{1/2}}.
\end{align*}

To prove the expansion~\eqref{eq:logphiexpand} we simply need to compute derivatives and evaluate at the origin. First, direct substitution gives $\log \oS(e^{i\bzer})= \log\oS(\one)= \log S(\one)$. The fact that all first order derivatives of $\log \oS$ vanish is a reflection of the fact that $(\one,1/S(\one))$ is a critical point. Explicitly, symmetry over the first $(d-1)$-coordinates implies that for each $1 \leq k \leq d-1$ we can write $\oS(\zz) = (z_k + \oz_k)B_k(\bzht{k})+Q(\bzht{k})$ so that 
\[ \oS_{z_k}(\zz) = (1 - z_k^{-2})B_k(\bzht{k}) \]
vanishes at $\zz=\one$, and 
\[ \oS_{z_d}(\zz) = A(\htbz) - z_d^{-2}B(\htbz) \]
also vanishes at $\zz=\one$ as $A(\one)=B(\one)$. The chain rule then implies that all first order partial derivatives of $(\log \oS)(e^{i\btheta})$ vanish at the origin. These expressions also imply that all mixed second order partial derivatives $\oS_{z_iz_j}(\one)$ for $i\neq j$ vanish, so all the mixed second order partial derivatives of $(\log \oS)(e^{i\btheta})$ vanish at the origin. Finally, we see that
\[ \oS_{z_kz_k}(\one) = 2 B_k(\one) = 2b_k \]
for each $1 \leq k \leq d-1$, and 
\[ \oS_{z_dz_d}(\one) = 2 B(\one) = 2b_d, \]
so that the chain rule gives
\[ \left(\frac{\partial^2}{\partial \theta_k^2} \log S(e^{i\btheta})\right)(\bzer) = \frac{\oS_{z_k}(\one)^2}{S(\one)^2} -\frac{\oS_{z_k}(\one)}{S(\one)} -\frac{\oS_{z_kz_k}(\one)}{S(\one)} = - 2 \frac{b_j}{S(\one)} \]
for all $1 \leq k \leq d$. The claimed expansion~\eqref{eq:logphiexpand} then holds as the quadratic term in the expansion is half the sum of the second order derivatives.
\end{proof}

\begin{rem}
Although Proposition~\ref{prop:leadingasm} only concerns the point $\one \in \Gamma$, there is an expansion
\[ \log \oS(w_1e^{i\theta_1},\dots,w_de^{i\theta_d}) = \log S(\ww) - \sum_{j=1}^{d}\frac{ b_j }{S(\bone)}\theta_j^2 + O ( (\theta_1 + \cdots + \theta_d)^3 ) \] 
for all $\ww \in \Gamma$. Indeed, for $\ww$ to lie in $\Gamma$ all terms $\ww^{\bi}$ for $\bi\in\mS$ must have the same argument, which can be used to show the quadratic term is constant for these expansions (this matters for computing higher-order terms in the asymptotic expansion of $s_n$). 
\end{rem}

To prove Theorem~\ref{thm:main} it is now sufficient to show that the only point of $\Gamma$ contributing to dominant asymptotics of~$s_n$ is $\one$.

\begin{proposition}
Under our running assumptions on $\epsilon$ and $\delta$, if $\ww\in\Gamma$ and $\ww \neq \one$ then
\[ \int_{\mN_\ww} \frac{(1+z_1)\cdots(1+z_{d-1})}{B\left(\htbz\right)(z_1\cdots z_d)} \cdot \frac{B\left(\htbz\right) - z_d^2 A\left(\htbz\right)}{1-z_d} \cdot \oS(\zz)^n \, d\zz = o(S(\bone)^n n^{-d/2}). \]
\end{proposition}

\begin{proof}
For a general point $\ww\in\Gamma$ the integral under consideration can be parameterized by $z_j=w_je^{i\theta_j}$ for $1 \leq j \leq d-1$ and $z_d=(1-\epsilon)w_de^{i\theta_d}$ to yield
\[ i^d\int_{\mM} P(\btheta)Q(\btheta) e^{n \phi(\btheta)} d\btheta,  \]
 where
\begin{align*}
P(\btheta) &= \frac{(1+w_1e^{i\theta_1})\cdots(1+w_{d-1}e^{i\theta_{d-1}})}{B\left(\widehat{\bw}e^{i\hat{\btheta}}\right)} \\[+2mm]
Q(\btheta) &= \frac{B\left(\widehat{\bw}e^{i\hat{\btheta}}\right) - (1-\epsilon)^2w_d^2e^{2i\theta_d} A\left(\widehat{\bw}e^{i\hat{\btheta}}\right)}{1-(1-\epsilon)w_de^{i\theta_d}}
\end{align*}
for $\widehat{\bw}e^{i\hat{\btheta}} = (w_1e^{i\theta_1},\dots,w_{d-1}e^{i\theta_{d-1}})$. Roughly speaking, the theory of ACSV (and saddle-point integrals more generally) shows that higher-order vanishing of $P$ and $Q$ corresponds to slower growing asymptotic behaviour. 

First, if $w_d=1$ and $w_j \neq 1$ for some $1 \leq j \leq d-1$ then $P$ vanishes at the origin, and following the argument used to prove Proposition~\ref{prop:leadingsaddle} shows that the contribution of $\ww$ is 
\[ \int_{\mM} P(\btheta)Q(\btheta) e^{n \phi(\btheta)} d\btheta = o\left(\int_{\zz \in \mM} e^{n  \log \oS\left(e^{i\btheta}\right)} d\btheta\right) = o( S(\bone)^n n^{-d/2}).\] 
Similarly, if $w_d=-1$ then $Q(\ww)=O(\epsilon)$ as the denominator of $Q$ no longer approaches zero as $\epsilon\rightarrow0$, but its numerator still does, and we can again show the contribution of $\ww$ lies in $o( S(\bone)^n n^{-d/2})$. In fact, because only one factor of $H$ vanishes at $\ww$ this asymptotic contribution can be computed using the theory of smooth ACSV to see that it lies in $O( S(\bone)^n n^{-(d+1)/2})$. 

Finally, suppose $w_d = \pm i$. The definition of $\Gamma$ implies that $w_dA(\ww)$ and $\ow_dB(\ww)$ have the same argument for any $\ww\in\Gamma$, so when $w_d=\pm i$ it must be the case that $A(\ww)=-B(\ww)$. This can only hold if $w_j = -1$ for at least one $1 \leq j \leq d-1$, so again the numerator $G$ and only one factor of $H$ vanishes at $\ww$, and the asymptotic contribution of $\ww$ is $O( S(\bone)^n n^{-(d+1)/2}) = o( S(\bone)^n n^{-d/2})$.
\end{proof}

\section{Conclusion}
\label{sec:conclusion}

Theorem~\ref{thm:main} completes the determination of dominant asymptotics for mostly symmetric short step models in an orthant. However, unlike the results for highly symmetric and non-zero drift mostly symmetric models, our arguments for the zero drift mostly symmetric case do not allow one to directly compute higher-order terms in an asymptotic expansion for the number of walks. Carefully tracing through the errors terms that appear in the analysis shows that we can satisfy our running assumptions on $\epsilon$ and $\delta$ and take them arbitrarily close to $n^{-1/2}$, resulting in an error term of the form $O(S(\one)^nn^{-\gamma})$ for any $\gamma<1/2$. 

Interestingly, there are models whose counting sequences have asymptotic series in $n^{-1/2}$ instead of series in $n^{-1}$ as usually occurs from ACSV arguments. To compute higher-order contributions, it is necessary to modify the approach of Proposition~\ref{prop:leadingsaddle} to account for higher-order terms involving $\epsilon$. Practically, one shifts the $d$th coordinate off the real axis and approximates a saddle-point-like integral whose amplitude is singular at the saddle points. A formalization of this approach for general ACSV arguments in \emph{non-generic directions} is postponed to forthcoming work (see also~\cite{BaryshnikovMelczerPemantle2023} for a special case considering rational functions whose denominators have only linear factors), but we end our discussion here with a few examples.

\begin{example}
\label{ex:zerodrift2}
As described in Example~\ref{ex:zerodrift}, asymptotic behaviour for the number of walks of length $n$ in the quadrant model defined by the step set $\mS = \{(0,1),(0,1),(-1,-1),(1,-1)\}$ (whose North step appears twice) is determined, up to a negligible error, by the integral
\[ \frac{1}{(2 \pi i)^2} \int_{\mN(1,1-\epsilon_1)} \frac{1+x}{2x^2y} \cdot \frac{2x - y^2(1+x^2)}{1-y} \oS(x,y)^n dx dy . \]
Restricting the domain of integration to a neighbourhood of the origin in a precise way allows us to take a series expansion and, up to a negligible error, approximate this integral by
\[ \frac{1}{(2 \pi)^2} \int_{\mM + i(0,\epsilon)} \left(4 + \frac{is^2}{t} + \frac{O((s+t)^3)}{t}\right) e^{-n\big(s^2/4 + t^2/2 + O((s+t)^3) \big)} d s d t , \]
where $\mM$ is a neighbourhood of the origin in $\R^2$. The integral of the leading term here is the same as above, however some additional work allows us to compute the second-order term and get
\begin{align*}
    &\frac{1}{(2 \pi)^2} \int_{\mM^2 + i(0,\epsilon)} \left(4 + \frac{is^2}{t} + \frac{O((s+t)^3)}{t}\right) e^{-n\big(s^2/4 + t^2/2 + O((s+t)^3) \big)} d s d t \\[+2mm]
    &= \frac{1}{(2 \pi)^2} \int_{\R^2 + i(0,\epsilon)} \left(4 + \frac{is^2}{t}\right) e^{-n\big(s^2/4 + t^2/2 \big)} d s d t + O\left(\frac{4^n}{n^2}\right) \\[+2mm]
    &= \frac{4^n}{n} \left(\frac{2\sqrt{2}}{\pi} + \frac{1}{\sqrt{\pi}}n^{-1/2} + O(n^{-1})\right). 
\end{align*}
Moving to three-dimensions, similar arguments show that the number of orthant walks of length $n$ in the mostly symmetric zero drift model with step set 
\[\mS_1 = \{ (1,0,-1), (-1,0,-1), (0,1,-1), (0,-1,-1) , (1,1,1), (-1,1,1), (1,-1,1), (-1,-1,1)\}\]
is given (up to negligible error) by
\begin{align*}
    &\frac{8^n}{(2 \pi)^3} \int_{\R^3 + i(0,0,\epsilon)} \left( 8 - \frac{i s^2}{u} - \frac{i t^2}{u} \right) e^{-n (3/8s^2 + 3/8t^2 + u^2/2)} ds dt du + O \left( \frac{8^n}{n^{5/2}} \right) \\[+2mm]
    &= \frac{8^n}{n^{3/2}} \left( \frac{8 \sqrt{2}}{3 \pi^{3/2}} - \frac{8}{9 \pi \sqrt{n}} + O(n^{-1}) \right).
\end{align*}
Finally, the number of orthant walks of length $n$ in the mostly symmetric zero drift model with step set  
\[ \mS = \{ ( \pm 1,0,1) , (0 , \pm 1, -1)  \} \]
is approximated by
\[ \frac{4^n}{(2 \pi )^3} \int_{\R^3 + i(0,0 , \epsilon)} \left( 8 + \frac{2is^2}{u} - \frac{2it^2}{u} \right)e^{-n (s^2/4 + t^2/4 + u^2/2)} ds dt du + O\left( \frac{4^n}{n^{5/2}} \right).\]
In this case, the fact that the terms in $\frac{2is^2}{u} - \frac{2it^2}{u}$ are symmetric and appear with opposite signs means that the second-order term in the asymptotic expansion drops out, giving 
\[ \frac{4^n}{n^{3/2}} \left( \frac{4 \sqrt{2}}{\pi^{3/2}} + O(n^{-1}) \right) \]
as the dominant asymptotic behaviour.
\end{example}

\section{Acknowledgements}

SM was partially funded by NSERC Discovery Grant RGPIN-2021-02382.

\bibliographystyle{alpha}
\bibliography{references}

\newcommand{\etalchar}[1]{$^{#1}$}
\begin{thebibliography}{BCvH{\etalchar{+}}17}

\bibitem[BBMR21]{BernardiBousquet-MelouRaschel2021}
Olivier Bernardi, Mireille Bousquet-M\'{e}lou, and Kilian Raschel.
\newblock Counting quadrant walks via {T}utte's invariant method.
\newblock {\em Comb. Theory}, 1:Paper No. 3, 77, 2021.

\bibitem[BCvH{\etalchar{+}}17]{BostanChyzakHoeijKauersPech2017}
A.~Bostan, F.~Chyzak, M.~van Hoeij, M.~Kauers, and L.~Pech.
\newblock Hypergeometric expressions for generating functions of walks with
  small steps in the quarter plane.
\newblock {\em European J. Combin.}, 61:242--275, 2017.

\bibitem[BK09]{BostanKauers2009}
Alin Bostan and Manuel Kauers.
\newblock Automatic classification of restricted lattice walks.
\newblock In {\em 21st {I}nternational {C}onference on {F}ormal {P}ower
  {S}eries and {A}lgebraic {C}ombinatorics ({FPSAC} 2009)}, Discrete Math.
  Theor. Comput. Sci. Proc., AK, pages 201--215. Assoc. Discrete Math. Theor.
  Comput. Sci., Nancy, 2009.

\bibitem[BM16]{Bousquet-Melou2016}
M.~Bousquet-M\'elou.
\newblock An elementary solution of {G}essel's walks in the quadrant.
\newblock {\em Adv. Math.}, 303:1171--1189, 2016.

\bibitem[BM23]{Bousquet-Melou2023}
Mireille Bousquet-M\'{e}lou.
\newblock Enumeration of three-quadrant walks via invariants: some diagonally
  symmetric models.
\newblock {\em Canad. J. Math.}, 75(5):1566--1632, 2023.

\bibitem[BMP23]{BaryshnikovMelczerPemantle2023}
Yuliy Baryshnikov, Stephen Melczer, and Robin Pemantle.
\newblock Asymptotics of multivariate sequences iv: Generating functions with
  poles on a hyperplane arrangement.
\newblock {\em Annals of Combinatorics}, 2023.

\bibitem[BRS14]{BostanRaschelSalvy2014}
Alin Bostan, Kilian Raschel, and Bruno Salvy.
\newblock Non-{D}-finite excursions in the quarter plane.
\newblock {\em J. Combin. Theory Ser. A}, 121:45--63, 2014.

\bibitem[DH21]{DreyfusHardouin2021}
Thomas Dreyfus and Charlotte Hardouin.
\newblock Length derivative of the generating function of walks confined in the
  quarter plane.
\newblock {\em Confluentes Math.}, 13(2):39--92, 2021.

\bibitem[DHRS21]{DreyfusHardouinRoquesSinger2021}
Thomas Dreyfus, Charlotte Hardouin, Julien Roques, and Michael~F. Singer.
\newblock On the kernel curves associated with walks in the quarter plane.
\newblock In {\em Transcendence in algebra, combinatorics, geometry and number
  theory}, volume 373 of {\em Springer Proc. Math. Stat.}, pages 61--89.
  Springer, Cham, 2021.

\bibitem[DLM16]{DArcoLacivitaMustapha2016}
Philippe D'Arco, Valentina Lacivita, and Sami Mustapha.
\newblock Combinatorics meets potential theory.
\newblock {\em Electron. J. Combin.}, 23(2):Paper 2.28, 17, 2016.

\bibitem[DT21]{DreyfusTrotignon2021}
Thomas Dreyfus and Am\'{e}lie Trotignon.
\newblock On the nature of four models of symmetric walks avoiding a quadrant.
\newblock {\em Ann. Comb.}, 25(3):617--644, 2021.

\bibitem[DW15]{DenisovWachtel2015}
Denis Denisov and Vitali Wachtel.
\newblock Random walks in cones.
\newblock {\em Ann. Probab.}, 43(3):992--1044, 2015.

\bibitem[FIM17]{FayolleIasnogorodskiMalyshev2017}
Guy Fayolle, Roudolf Iasnogorodski, and Vadim Malyshev.
\newblock {\em Random walks in the quarter plane}, volume~40 of {\em
  Probability Theory and Stochastic Modelling}.
\newblock Springer, Cham, second edition, 2017.

\bibitem[Hum10]{Humphreys2010}
Katherine Humphreys.
\newblock A history and a survey of lattice path enumeration.
\newblock {\em J. Statist. Plann. Inference}, 140(8):2237--2254, 2010.

\bibitem[KKZ09]{KauersKoutschanZeilberger2009}
Manuel Kauers, Christoph Koutschan, and Doron Zeilberger.
\newblock Proof of {I}ra {G}essel's lattice path conjecture.
\newblock {\em Proc. Natl. Acad. Sci. USA}, 106(28):11502--11505, 2009.

\bibitem[KM03]{KrattenthalerMohanty2003}
C.~Krattenthaler and S.~G. Mohanty.
\newblock Lattice path combinatorics - applications to probability and
  statistics.
\newblock In {\em Encyclopedia of Statistical Sciences, Second Edition}. Wiley,
  New York, 2003.

\bibitem[Mel21]{Melczer2021}
Stephen Melczer.
\newblock {\em An Invitation to Analytic Combinatorics: From One to Several
  Variables}.
\newblock Texts and Monographs in Symbolic Computation. Springer International
  Publishing, 2021.

\bibitem[MM16]{MelczerMishna2016}
S.~Melczer and M.~Mishna.
\newblock Asymptotic lattice path enumeration using diagonals.
\newblock {\em Algorithmica}, 75(4):782--811, 2016.

\bibitem[Moh79]{Mohanty1979}
S.~G. Mohanty.
\newblock {\em Lattice path counting and applications}.
\newblock Academic Press [Harcourt Brace Jovanovich, Publishers], New
  York-London-Toronto, Ont., 1979.

\bibitem[MW19]{MelczerWilson2019}
Stephen Melczer and Mark~C. Wilson.
\newblock Higher dimensional lattice walks: Connecting combinatorial and
  analytic behavior.
\newblock {\em SIAM J. Discrete Math.}, 33(4):2140--2174, 2019.

\bibitem[Nar79]{Narayana1979}
T.~V. Narayana.
\newblock {\em Lattice path combinatorics with statistical applications},
  volume~23 of {\em Mathematical Expositions}.
\newblock University of Toronto Press, Toronto, Ont., 1979.

\bibitem[Pri22]{Price2022}
Andrew~Elvey Price.
\newblock Enumeration of walks with small steps avoiding a quadrant.
\newblock {\em S\'{e}m. Lothar. Combin.}, 86B:Art. 1, 12, 2022.

\bibitem[PWM24]{PemantleWilsonMelczer2024}
Robin Pemantle, Mark~C. Wilson, and S.~Melczer.
\newblock {\em Analytic combinatorics in several variables, 2nd Edition}.
\newblock In press, Cambridge Studies in Advanced Mathematics. Cambridge
  University Press, Cambridge, 2024.

\bibitem[Ras12]{Raschel2012}
Kilian Raschel.
\newblock Counting walks in a quadrant: a unified approach via boundary value
  problems.
\newblock {\em J. Eur. Math. Soc. (JEMS)}, 14(3):749--777, 2012.

\bibitem[RT19]{RaschelTrotignon2019}
Kilian Raschel and Am\'{e}lie Trotignon.
\newblock On walks avoiding a quadrant.
\newblock {\em Electron. J. Combin.}, 26(3):Paper No. 3.31, 34, 2019.

\bibitem[Tro22]{Trotignon2022}
Am\'{e}lie Trotignon.
\newblock Discrete harmonic functions in the three-quarter plane.
\newblock {\em Potential Anal.}, 56(2):267--296, 2022.

\end{thebibliography}

\end{document}